\documentclass[12pt]{amsart}

\usepackage{amsmath}
\usepackage{amssymb}
\usepackage{amsthm}
\usepackage{amscd}

\usepackage{color}
\usepackage{hyperref}

\usepackage[latin1]{inputenc}

\usepackage{mathpazo}
\usepackage[scaled=.95]{helvet}
\usepackage{courier}

\swapnumbers
\headsep=1cm \numberwithin{equation}{section}
\newtheorem{theorem}{Theorem}[section]
\newtheorem{lemma}[theorem]{Lemma}
\newtheorem{proposition}[theorem]{Proposition}
\newtheorem{corollary}[theorem]{Corollary}

\theoremstyle{definition}
\newtheorem{definition}[theorem]{Definition}

\newtheorem{remark}[theorem]{Remark}
\newtheorem{remark and definition}[theorem]{Remark and Definition}
\newtheorem{remark and notation}[theorem]{Remark and Notation}

\newtheorem{notation}[theorem]{Notation}
\newtheorem{example}[theorem]{Example}

\newtheorem{question}[theorem]{Question}

\newcommand\Hom{\operatorname{Hom}}
\newcommand\Ext{\operatorname{Ext}}
\newcommand\Tor{\operatorname{Tor}}

\newcommand\depth{\operatorname{depth}}
\newcommand\grade{\operatorname{grade}}

\newcommand\Ker{\operatorname{\Ker}}

\newcommand\pd{\operatorname{pd}}
\newcommand\id{\operatorname{id}}
\newcommand\Supp{\operatorname{Supp}}

\newcommand\Ass{\operatorname{Ass}}
\newcommand\Ann{\operatorname{Ann}}

\newcommand\cd{\operatorname{cd}}

\newcommand{\xx}{\underline x}

\newcommand{\qism}{\stackrel{\sim}{\longrightarrow}}

\newcommand\RHom{\operatorname{R Hom}}
\newcommand{\Rgam}{{\rm R} \Gamma_{\mathfrak{a}}}
\newcommand{\RGam}{{\rm R} \Gamma_{\mathfrak{m}}}

\author[Freitas,\, Jorge-P\'erez,\, Miranda-Neto,\, Schenzel]{Thiago H. Freitas,\, Victor H. Jorge-P\'erez,\, Cleto B. Miranda-Neto,\, Peter Schenzel}

\title[Generalized Local Cohomology]{Generalized local duality, canonical modules, and prescribed bound on projective dimension}

\address{Universidade Tecnol\'ogica Federal do Paran\'a, 85053--525, Guarapuava-PR, Brazil}
\email{freitas.thf@gmail.com}

\address{Universidade de S{\~a}o Paulo -
ICMC, Caixa Postal 668, 13560-970, S{\~a}o Carlos-SP, Brazil}
\email{vhjperez@icmc.usp.br}

\address{Universidade Federal da Para\'iba - 58051-900, Jo\~ao Pessoa, PB, Brazil}
\email{cleto@mat.ufpb.br}

\address{Martin-Luther-Universit\"at Halle-Wittenberg, Institut f\"ur Informatik, D--06 099 Halle (Saale), Germany}
\email{schenzel@informatik.uni-halle.de}

\date{\today}

\thanks{{\it Corresponding author}: C. B. Miranda-Neto (cleto@mat.ufpb.br).}


\begin{document}
\begin{abstract} We present various approaches to J. Herzog's theory of generalized local cohomology and explore its main aspects, e.g., (non-)vanishing results as well as a general local duality theorem which extends, to a much broader class of rings, previous results by Herzog-Zamani and Suzuki. As an application, we establish a prescribed upper bound for the projective dimension of a module satisfying suitable cohomological conditions, and we derive some freeness criteria and questions of Auslander-Reiten type. Along the way, we prove a new characterization of Cohen-Macaulay modules which truly relies on generalized local cohomology, and in addition we introduce and study a generalization of the notion of canonical module.

\end{abstract}

\maketitle

\section*{Introduction}

Our goal in this work is to develop approaches and improvements to the theory of generalized local cohomology initiated by Herzog \cite{jH}, and further developed by Herzog and Zamani \cite{HZ}, Suzuki \cite{nS}, and Yassemi \cite{yassemi}. Our 
main result extends a well-known generalized local duality theorem, stated originally over Cohen-Macaulay local rings with canonical module, to the much broader class of quotients of Gorenstein local rings (equivalently, rings possessing a dualizing complex). In addition, we introduce a generalization of the concepts of deficiency and canonical modules (cf.\,\cite[Section 1]{pS}), which is of interest on its own and in particular will serve as a convenient tool to some of our results.

Regarding applications, we focus mainly on establishing a prescribed bound for the (assumed finite) projective dimension of a module over a local ring possessing a dualizing complex and satisfying certain cohomological conditions. A bit more precisely, under suitable hypotheses and considering an integer $i\geq 1$ attached to the vanishing of finitely many appropriate cohomology modules, we derive that $i-1$ is an upper bound for the projective dimension of the given module. As a consequence, we detect some freeness criteria and raise questions in the spirit of the famous Auslander-Reiten conjecture, proposed about 45 years ago in  \cite{auslander} and which has been the objective of renewed attention by many authors (see, e.g., \cite{araya}, \cite{sgoto}, and \cite{hule}). Along the way, we also detect a new characterization of Cohen-Macaulay modules, in terms of the non-vanishing of a suitable local cohomology module; this is a legitimate application of generalized local cohomology theory as the result cannot be recovered by means of the ordinary tool (nor by any existing duality theorem), and therefore it constitutes, in our view, a particular compelling reason to consider broader forms of the classical theory.

In order to describe our main object of study, let us fix a convention that will be in force throughout the paper: by {\it ring} we shall tacitly mean a Noetherian commutative ring with identity $1\neq 0$. Now let  $\mathfrak{a}$ be an ideal of a ring $R$, and let $N$ be an $R$-module. For an integer $i\geq 0$, the $i$-th local cohomology module
$H^{i}_\mathfrak{a}(N)=\displaystyle \varinjlim{\rm Ext}_R^{i}(R/\mathfrak{a}^n, N)$ of $N$ with respect to $\mathfrak{a}$ has been -- needless to say -- widely studied with a view to a variety of important applications in commutative algebra and algebraic geometry. For details about this classical object, see \cite{aG}, also \cite{B-Sharp}, \cite{24h}. A natural generalization of this notion was introduced in Herzog  \cite{jH}. Precisely, if in addition we consider an $R$-module $M$, then the {\it $i$-th generalized local cohomology module of $M$, $N$ with respect to $\mathfrak{a}$} is $$H^{i}_\mathfrak{a}(M,N) \, = \, \displaystyle \varinjlim_{n}{\rm Ext}_R^{i}(M/\mathfrak{a}^nM,N).$$ Notice that the case $M=R$ retrieves ordinary local cohomology. Essentially, this paper presents new directions in the study of such modules, providing applications that in particular further justify their investigation.

Now let us briefly describe the contents of the paper.

Section 1 gives the central definitions 
and presents the first general results. For instance, Lemma \ref{thm-1} provides a number of quasi-isomorphic complexes yielding various ways to realize generalized local cohomology (see also Example \ref{cor-1}). Furthermore, under a suitable hypothesis, Corollary \ref{cor-3} describes $H^{i}_\mathfrak{a}(M,N)$ as an Ext module; this is used in Remark \ref{fails} to illustrate that the independence property (regarding base change) fails for generalized local cohomology modules.

Section 2 furnishes three spectral sequences as fundamental tools that allow for the determination of generalized local cohomology (see Proposition \ref{spec-1}), which in turn led us to the detection of useful properties -- e.g., (non-)vanishing results -- such as the ones given in Corollaries \ref{cor-4}, \ref{cor-5} and \ref{cor-6}.

Section 3 studies Cohen-Macaulay modules and generalized local duality.
For instance, Theorem \ref{theo01} characterizes cohomologically when the local ring $R$ (possessing a dualizing complex) and a given finitely generated $R$-module of finite projective dimension have the same Cohen-Macaulay defect. As a consequence, Corollary \ref{CM-charact} gives the above-mentioned new characterization of Cohen-Macaulay modules. The main result of the section, Theorem \ref{thm-3}, concerns duality and extends Suzuki's local duality theorem (cf.\,\cite{nS}). The case where the ring is Cohen-Macaulay is recorded in Corollaries \ref{cor} and \ref{cor1}, which in particular retrieve the local duality theorem of Herzog and Zamani \cite{HZ}.

Section 4 introduces generalized modules of deficiency and, particularly, the concept of generalized canonical module. Such tools play a role in conveniently describing generalized local cohomology modules in the Gorenstein case. Structural features such as the behavior of the generalized canonical module under localization (Proposition \ref{prop-3}), associated primes (Proposition \ref{lem-1}), and depth (Proposition \ref{prop-5}) are also investigated.

Finally, Section 5 describes the already mentioned method to produce, as a byproduct of our generalized local duality theorem, a prescribed upper bound for the (finite) projective dimension of a finitely generated module under suitable cohomological conditions (Theorem \ref{main-result} and Corollary \ref{main-cor}). In Remark \ref{Jo-rem} we observe a connection to Jorgensen \cite[Question 2.7]{jor} and we suggest a generalization of the problem in Question \ref{Jo-ques}. We then turn our attention to deriving freeness criteria, the main ones being stated in Corollaries \ref{cor1-free} and \ref{cor2-free}. Such criteria are of Auslander-Reiten type in the sense of bearing a visible similarity with the classical Auslander-Reiten conjecture (Remark \ref{AR-rem}). This motivated us to propose a new variant of the problem (see Question \ref{AR-ques}), which matches the original statement in the Gorenstein case.

\section{Definitions and First Results}

As anticipated in the introduction, by {\it ring} we always mean Noetherian commutative ring with $1\neq 0$.

Let $R$ denote a ring and let $\mathfrak{a}$ be an ideal of $R$ generated by the set $\xx = x_1,\ldots,x_k$. In addition, let $M,N$ be two
$R$-modules. As a  generalization of the local cohomology modules
$H^i_{\mathfrak{a}}(M)$, $i \geq 0,$ defined as the right derived functors of the
section functor $\Gamma_{\mathfrak{a}}(\cdot)$, i.e., $\Gamma_{\mathfrak{a}}(M)=
\{m\in M \mid \mathfrak{a}^sm=0 \text{ for some } s>0\}$ (see \cite{aG}, also \cite{B-Sharp}, \cite{24h}), Herzog introduced the following
tool (see \cite{jH}).

\begin{definition} \label{def-1}
	The {\it $i$-th generalized local cohomology module of $M,N$ with respect to $\mathfrak{a}$} is
	\[
	H^i_{\mathfrak{a}}(M,N) = \varinjlim_{n} \Ext^i_R(M/\mathfrak{a}^nM,N),
	\]
	where, for each $i\geq 0$, the direct system is induced by the natural maps $M/\mathfrak{a}^{n+1}M
	\to M/\mathfrak{a}^nM$.
\end{definition}

Taking
$M=R$ we readily get $H^i_{\mathfrak{a}}(N) = \varinjlim \Ext_R^i(R/\mathfrak{a}^n, N)= H^i_{\mathfrak{a}}(R, N)$, for all $i \geq 0$.

A first systematic approach has been done by
Suzuki (see \cite{nS}). Here we shall add some results and simplify the proof of
a few known facts. As a first vanishing result we have the following,
shown by Bijan-Zadeh (see \cite{BZm}).

\begin{proposition} \label{prop-1}
	Let $\mathfrak{a}$ be an ideal of a local ring $R$. Let
	$M,N$ be two finitely generated $R$-modules, and set\, $s = \grade(\mathfrak{a}+\Ann_R M,\,N)$. Then
	\[
	H^i_{\mathfrak{a}}(M,N) = 0 \quad \mbox{ for } \quad i <s, \quad \mbox{ and } \quad H^s_{\mathfrak{a}}(M,N) \not= 0.
	\]
	
\end{proposition}

\begin{proof}
	Let $I = \mathfrak{a}+\Ann_R M$. Then
	\[
	\grade(I,N) = \inf \{i \  |\Ext^i_R(M/\mathfrak{a}^nM,N) \not= 0\}
	\]
	for all $n \geq 1$, because $\Supp_R M/\mathfrak{a}^nM = V(\mathfrak{a})$ (see \cite[1.2.10]{BH}).
	This proves the vanishing part of the claim.
	
	Now by the short exact sequence
	\[
	0 \to \mathfrak{a}^nM/\mathfrak{a}^{n+1} M \to M/\mathfrak{a}^{n+1}M
	\to M/\mathfrak{a}^n M \to 0,
	\]
	there is an injection $0 \to \Ext_R^s(M/\mathfrak{a}^nM,N) \to
	\Ext_R^s(M/\mathfrak{a}^{n+1}M,N)$. Note that $$\Ext_R^i(\mathfrak{a}^nM/\mathfrak{a}^{n+1} M,N) =0 \mbox{ for } i <s.$$
	That proves the non-vanishing of the direct limit $H^s_{\mathfrak{a}}(M,N)$.
\end{proof}

In the following we have a more functorial, alternative look at generalized local 
cohomology. To this end we use some definitions of derived functors 
(see, e.g., \cite{SS} and the references there). 
\begin{definition} \label{defnew-1}
	For $R$-modules $M, N$, an ideal $\mathfrak{a}$ and $i \in \mathbb{Z}$, we define $\widetilde{H}^i_{\mathfrak{a}}(M,N)$
	as the $i$-th cohomology of $\Rgam(\RHom_R(M, N))$. As will be explained shortly, $H^i_{\mathfrak{a}}(M,N)\cong \widetilde{H}^i_{\mathfrak{a}}(M,N)$ whenever $M$ is finitely generated. Moreover, there is no problem in extending the definition to two complexes
	in place of the two modules; we shall use this only in the second place for the dualizing complex, which is a bounded complex of injective modules with finitely generated cohomology modules. 
\end{definition}

For technical details about quasi-isomorphisms, as well as notations and related definitions, we refer to \cite{SS}.

\begin{notation} 
	(A) Let $R$ be a ring and consider the ideal  $\mathfrak{a} = \xx R$ generated by a 
	system of elements $\xx$ of $R$. Let $K_{\bullet}(\xx)$ denote the associated Koszul complex. For an integer $n \geq 1$, we put $\xx^n = x_1^n,\ldots,x_k^n$.
	We denote by $\check{C}_{\xx}$ the \v{C}ech complex of $R$ with respect
	to $\xx$. Note that $\check{C}_{\xx} \cong \varinjlim K_{\bullet}(\xx^n)$ with the natural
	maps of Koszul complexes. \\
	(B)	Let $M, N$ denote two $R$-modules.
	Let $F_{\bullet} \qism M$ be a free resolution of $M$, and $N \qism I^{\bullet}$
	be an injective resolution of $N$.
\end{notation}

For our use of generalized local cohomology, we have the following.

\begin{lemma} \label{thm-1}
	{\rm (A)} With the previous notations, $\Rgam(\RHom_R(M,N))$ has the following representatives: 
	\begin{itemize}
		\item[(a)] $\Gamma_{\mathfrak{a}}(\Hom_R(F_{\bullet},I^{\bullet}))$,
		\item[(b)] $\check{C}_{\xx} \otimes_R \Hom_R(F_{\bullet},N)$, and 
		\item[(c)] $\check{C}_{\xx} \otimes_R \Hom_R(M,I^{\bullet})$.
	\end{itemize}
{\rm (B)} If in addition $M$ is finitely generated and $F_{\bullet}$ is a resolution of $M$
	by finitely generated free $R$-modules, then $\Rgam(\RHom_R(M,N))$ has the additional representatives 
	\begin{itemize}
		\item[(d)] $\varinjlim \Hom_R(M/\mathfrak{a}^nM,I^{\bullet})$, and
		\item[(e)] $\Hom_R(M, \Gamma_{\mathfrak{a}}(I^{\bullet})) \cong \RHom_R(M,\Rgam(N))$.
	\end{itemize}
\end{lemma}

\begin{proof}
	Because $\check{C}_{\xx}$ is a bounded complex of flat $R$-modules, and by virtue of 
	the quasi-isomorphism (see \cite[Proposition 7.4.1]{SS} or \cite{jL}) 
	\[
	\Gamma_{\mathfrak{a}}(\Hom_R(F_{\bullet},I^{\bullet})) \qism
	\check{C}_{\xx} \otimes_R \Hom_R(F_{\bullet},I^{\bullet}),
	\]
	the statement in (A) follows.
	
	For the proof of (B),  recall that 
	\[
	\Gamma_{\mathfrak{a}}(\Hom_R(F_{\bullet},I^{\bullet})) \cong
	\Hom_R(F_{\bullet},\Gamma_{\mathfrak{a}}(I^{\bullet})) \qism 
	\textstyle{\varinjlim\limits_n} \Hom_R(M/\mathfrak{a}^nM,I^{\bullet})
	\]
	as follows easily by the definitions and adjointness. 
\end{proof}

In the case where $M$ is finitely generated, the isomorphism in \ref{thm-1}(e)
implies that the definitions of $H^i_{\mathfrak{a}}(M,N)$ and $\widetilde{H}^i_{\mathfrak{a}}(M,N)$ given respectively in 
\ref{def-1} and \ref{defnew-1} are equivalent, i.e., these cohomology modules are isomorphic. It is worth observing that this is {\it not} the case in general. 

\begin{example} \label{cor-1}
	Let $(R,\mathfrak{m})$ denote a Gorenstein complete local ring with $\dim R = d > 0$. Let $E =  E(R/\mathfrak{m})$ stand for the injective hull of the residue field. Then $\RHom_R(E,E) \cong R$ 
	and therefore  $\widetilde{H}^d_{\mathfrak{m}}(E,E) \cong E\neq 0$, and $\widetilde{H}^i_{\mathfrak{m}}(E,E) = 0$ for all $i \not= d$. 
	
	On the other hand, $$H^i_{\mathfrak{m}}(E,E) = \varinjlim_n \Ext_R^i(E/\mathfrak{m}^nE,E) = 0$$ for all $i$. This follows 
	for $i >0$ since $E$ is injective,  and for $i=0$ we have 
	$H^0_{\mathfrak{m}}(E,E) = \varinjlim \Hom_R(E/\mathfrak{m}^nE,E) \cong \Gamma_{\mathfrak{m}}(R) = 0$. 
\end{example}

Here we use mainly finitely generated $R$-modules $M$, the key advantage of the definition \ref{defnew-1} being that it is more functorial. 

Similar results for $H^i_{\mathfrak{a}}(M,N)$ to that above, sometimes under more restrictive assumptions,
have been shown by Herzog (see \cite[1.1.6]{jH}), and Divaani-Aazar, Sazeedeh, and Tousi (see \cite{DST}). 

By virtue of  the isomorphisms $H^i_{\mathfrak{a}}(M,N) \cong \widetilde{H}^i_{\mathfrak{a}}(M,N), 
	i \in \mathbb{N}$, for a finitely generated $R$-module $M$, the following corollary follows easily by \ref{thm-1}.

\begin{corollary} \label{cor-2}
	We use the previous notation.
	\begin{itemize}
		\item[(a)] Let $M$ be a finitely generated $R$-module and $0 \to N' \to N \to N'' \to 0$ be
		a short exact sequence of $R$-modules. Then there is a long exact sequence
		\[
		\ldots \to H^i_{\mathfrak{a}}(M,N') \to H^i_{\mathfrak{a}}(M,N) \to
		H^i_{\mathfrak{a}}(M,N'') \to H^{i+1}_{\mathfrak{a}}(M,N') \to \ldots
		\]
		\item[(b)] Let $0 \to M' \to M \to M'' \to 0$ be a short exact sequence
		of finitely generated $R$-modules and $N$ be an $R$-module. Then there is
		a long exact cohomology sequence
		\[
		\ldots \to  H^i_{\mathfrak{a}}(M'',N) \to H^i_{\mathfrak{a}}(M,N) \to
		H^i_{\mathfrak{a}}(M',N) \to H^{i+1}_{\mathfrak{a}}(M'',N) \to \ldots
		\]
	\end{itemize}
\end{corollary}

Another result says something about the degeneration of the generalized local
cohomology modules. For the proof of the next corollary, we use \ref{thm-1} and, e.g., \cite[3.2.1]{jL}.

\begin{corollary} \label{cor-3}
	Let $M$ be a finitely generated $R$-module.
	Suppose that $\Supp_RM \cap \Supp_R N \subseteq V(\mathfrak{a})$. Then
	there are isomorphisms $$H^i_{\mathfrak{a}}(M,N) \cong \Ext_R^i(M,N)$$
	for all $i \geq 0$.
\end{corollary}

\begin{remark}\label{fails}\rm It has to be pointed out that the independence property (concerning base change) fails for generalized local cohomology. To produce an example, let $x$ be an indeterminate over a field $\Bbbk$. Set $R= \Bbbk [[x]]$,  $\mathfrak{m}=(x)$ and $R'= \Bbbk [[x]]/ (x^2)$. Clearly, $R'/\mathfrak{m}R' \cong R/\mathfrak{m} \cong \Bbbk$ as $R$-modules, and $\Bbbk$ is also an $R'$-module.
Since $R'$ is not regular, $\pd_{R'}(\Bbbk)=\infty$. Moreover, since
$$H^i_{\mathfrak{m}R'}( \Bbbk, \Bbbk) \, \cong \, {\rm Ext}^i_{R'}( \Bbbk, \Bbbk),$$ we conclude that  $H^i_{\mathfrak{m}R'}( \Bbbk, \Bbbk) \neq 0$ for infinitely many values of $i$. But   $H^i_{\mathfrak{m}}( \Bbbk, \Bbbk)\cong {\rm Ext}^i_{R}( \Bbbk, \Bbbk)=0 $ for all $i\geq 1$.
\end{remark}

\section{Computations of generalized local cohomology}

According to \ref{thm-1}, there are plenty of complexes allowing for the determination of $H^i_{\mathfrak{a}}(M,N)$ whenever $M$ is a finitely generated $R$-module, as they provide useful spectral sequences for this purpose. Under additional (finiteness) assumptions, such sequences are convergent and finite.

\begin{proposition} \label{spec-1} 
	Let $\mathfrak{a}$ denote an ideal of a ring $R$. Let $M, N$ be two $R$-modules, with $M$ finitely generated. Then there are the following spectral sequences:
    \begin{itemize}
    	\item[(a)] $E_2^{i,j} = H^i_{\mathfrak{a}}(\Ext_R^j(M,N))
    	\Longrightarrow  E_{\infty}^{i+j} = H^{i+j}_{\mathfrak{a}}(M,N)$. 
    	\item[(b)] $	E_2^{i,j} = \Ext_R^i(M,H_{\mathfrak{a}}^j(N))
    	\Longrightarrow  E_{\infty}^{i+j} =
    	H^{i+j}_{\mathfrak{a}}(M,N)$.
    	\item[(c)] $E_2^{i,j} = \Tor_j^R(H^i_{\mathfrak{a}}(M,R),N) \Longrightarrow
    	E_{\infty}^{i-j} = H^{i-j}_{\mathfrak{a}}(M,N)$, if $R$ is local and $\pd_RM < \infty$.
    \end{itemize}
\end{proposition}
\begin{proof} As it is explained by Hartshorne (see \cite{RD}) the spectral sequences grow out of the composite of derived functors. Hence (a) and (b) follow by \ref{thm-1}(a) and (e). Finally, recall that $$\Rgam(\RHom_R(M,N)) \cong \Rgam(\RHom_R(M,R))\otimes_R^{\rm{L}}N$$ for $M$ 
a finitely generated $R$-module of finite projective dimension. This provides the spectral sequence in (c).
\end{proof}

The previous spectral sequences may not be convergent. To circumvent this issue, we require suitable conditions in the next corollaries.

\begin{corollary} \label{cor-4}
	With the notations of \ref{spec-1}, assume that $M,N$ are finitely generated and put $d = \dim M \otimes_R N$.
	Suppose that $\Ext_R^i(M,N)= 0$ for all $i > p$. Then
	\[
		H^i_{\mathfrak{a}}(M,N) = 0 \quad \mbox{ for } \quad i > p+d, \quad \mbox{ and } \quad
		H^{p+d}_{\mathfrak{a}}(M,N) \cong H^d_{\mathfrak{a}}(\Ext_R^p(M,N)).
	\]
\end{corollary}
\begin{proof}
	Since $\Ext_R^i(M,N) = 0$ for all $i > p$, the  spectral sequence of \ref{spec-1}(a)
	\[
	E_2^{i,j} = H^i_{\mathfrak{a}}(\Ext_R^j(M,N))
	\Longrightarrow  E_{\infty}^{i+j} =
	H^{i+j}_{\mathfrak{a}}(M,N)
	\]
	is now convergent. Note that $\dim \Ext_R^j(M,N) \leq d$ for all $j$.
	That is, $E_2^{i,j} = 0$ for all $j > d$ and all $i > p$. This shows that
	$E_{\infty}^{i+j} = 0$ whenever $i+j > p+d$.
	
	Next we investigate the situation $i+j = p+d$. The subsequent stages of the spectral
	sequence are
	\[
	 E^{i-r,j+r-1} \to  E^{i,j} \to  E^{i+r,j-r+1}
	\quad \mbox{ for } \quad r \geq 2.
	\]
	Now $ E^{i+r,j-r+1} $ is a subquotient of
	$E_2^{i+r,j-r+1} = 0$ for $r \geq 2$, while
	$ E^{i-r,j+r-1} = 0$ for $r \geq 2$. The last vanishing follows since $j+r-1 \leq d$ implies that
	$i \geq p+r-1$. Therefore $E_2^{p,d} =
	E_{\infty}^{p,d}$ and $E_{\infty}^{i,j} = 0$
	for all $(i,j) \not = (p,d)$ with $i+j = p+d$. This completes the proof.
\end{proof}

For the next result, recall the definition of the {\it cohomological dimension}
$$\cd(\mathfrak{a},N) = \sup \{i \in \mathbb{N}\, | \, H^i_{\mathfrak{a}}(N) \not= 0\}$$  of an $R$-module $N$ with respect to $\mathfrak{a}$. Note that $\cd(\mathfrak{a},N) \leq \dim_R N\leq \dim R$. The next result can also be found in \cite[Proposition 2.8]{HV}.

\begin{corollary} \label{cor-5}
	Let $R$ be a local ring and $\mathfrak{a}$ be an ideal of $R$. Let $M, N$ be two finitely generated $R$-modules. Suppose
	that $p:=\pd_R M  < \infty$, and let $n := \cd(\mathfrak{a},N)$.
	\begin{itemize}
		\item[(a)] $H^i_{\mathfrak{a}}(M,N) = 0$
		for all $i > p+n$.
		\item[(b)] $H^{p+n}_{\mathfrak{a}}(M,N)
		\cong \Ext_R^p(M,H^n_{\mathfrak{a}}(N)) \cong \Ext_R^p(M,R) \otimes_R H^n_{\mathfrak{a}}(N)$.
	\end{itemize}
\end{corollary}
\begin{proof}
	By \ref{spec-1}(b), there
	is the convergent spectral sequence
	\[
	E_2^{i,j} = \Ext_R^i(M,H^j_{\mathfrak{a}}(N))
	\Longrightarrow  E_{\infty}^{i+j} =
	H^{i+j}_{\mathfrak{a}}(M,N).
	\]
	First, we examine it for $i+j > p+n$. We have $E_2^{i,j} = 0$ because either $i > p$ or $j > n$. Note that $H^i_{\mathfrak{a}}(N) = 0$ for $i > \cd(\mathfrak{a},N)$.  This yields $E_{\infty}^{i+j} =0$, which proves the claim  (a).

	For the proof of (b), let $i+j = p+n$. Then the subsequent stages of the spectral sequence are
	\[
	 E^{i-r,j+r-1} \to  E^{i,j} \to  E^{i+r,j-r+1}
	\mbox{ for } r \geq 2.
	\]
	Now $ E^{i+r,j-r+1} $ is a subquotient of
	$E_2^{i+r,j-r+1} = 0$ for $r \geq 2$, while
	$ E^{i-r,j+r-1} = 0$ for $r \geq 2$. The last vanishing follows since $j+r-1 \leq n$ implies that $i \geq p+r-1$. Hence $E_2^{p,n} =
	E_{\infty}^{p,n}$ and $E_{\infty}^{i,j} = 0$
	for all $(i,j) \not = (p,n)$ with $i+j = p+n$.
	By the convergence of the spectral sequence,
	there is a partial degeneration and this proves (b).
\end{proof}

If $R$ is a local ring, there is a well-known uniform vanishing result for ordinary local cohomology, to wit, $H^i_{\mathfrak{a}}(\cdot) = 0$ for $i > \dim R$. The following remark
is important in view of generalized local cohomology.

\begin{remark} \label{rem-1}
	Let $(R,\mathfrak{m})$ be a local ring with residue field $\Bbbk$. Then
	$H^i_{\mathfrak{a}}(\cdot,N) = 0$  for all $i \gg 0$ if and only if $\id_R N <
	\infty$. This statement is clear. Now, if $M$ is a finitely generated $R$-module then we claim that $H^i_{\mathfrak{a}}(M, N) = 0$ for all $i \gg 0$, whenever $N$ is a finitely generated $R$-module, if and only if $\pd_R M < \infty$. To see this, choose $N = \Bbbk$;  hence
	$$H^i_{\mathfrak{a}}(M,\Bbbk) \cong \Ext_R^i(M,\Bbbk)$$ by \ref{cor-3} and then, from the hypothesis, $\Ext_R^i(M,\Bbbk)$ must vanish for $i \gg 0$, which implies that (and is clearly equivalent to) ${\rm pd}_RM< \infty$. Conversely, if $M$ has finite projective dimension, say $p$, then by \ref{cor-5}(a) we get $H^i_{\mathfrak{a}}(M, N)=0$ for all $i>p+{\rm cd}(\mathfrak{a}, N)$, whenever $N$ is a finitely generated $R$-module. But the cohomological dimension is bounded above by $t=\dim R$. Thus, precisely, $H^i_{\mathfrak{a}}(M, N)=0$ for all $i>p+t$.
\end{remark}

For some applications in the subsequent sections we shall need the following 
degenerations of the spectral sequences established in \ref{spec-1}(b) and (c).

\begin{corollary} \label{cor-6}
	Let $M, N$ be two $R$-modules, with $M$ finitely generated. 
	\begin{itemize}
		\item[(a)] Suppose $H^i_{\mathfrak{a}}(N) = 0$ for all $i \not= q$. Then, 
	$\Ext_R^{i-q}(M,H^q_{\mathfrak{a}}(N)) \cong H^i_{\mathfrak{a}}(M,N)$ for all $i \geq 0$, and in particular $H^i_{\mathfrak{a}}(M,N) = 0$
	for all $i < q$.
	\item[(b)] Suppose $R$ is local, $\pd_RM < \infty$ and $H^i_{\mathfrak{a}}(M,R) = 0$
	for all $i \not= q$. Then, there are isomorphisms
	$
	\Tor_j^R(H^q_{\mathfrak{a}}(M,R),N) \cong H^{q-j}_{\mathfrak{a}}(M,N)
	$
	for all $j \geq 0$.
	\end{itemize}. 
\end{corollary}

\section{Cohen-Macaulayness and Generalized Duality}

In this section we investigate Cohen-Macaulayness of modules and establish a generalized local duality theorem.

\begin{notation} \label{not-2} 
	(A) Let $(R,\mathfrak{m})$ be a local ring possessing a dualizing complex $D^{\bullet}$ (equivalently, $R$ is a quotient of a Gorenstein local ring). We tacitly regard $D^{\bullet}$  as a bounded complex of injective $R$-modules with finitely generated cohomology. Moreover, we suppose that $D^{\bullet}$ is {\it normalized} in the sense of \cite[11.4.6]{SS}. \\
	(B) (\textsl{Local Duality Theorem}) Let $R$ and $D^{\bullet}$ be as above, and let $X$ be an $R$-complex with finitely generated cohomology modules.  Then 
	\[
		\RGam (X) \cong \Hom_R(\RHom_R(X,D^{\bullet}),\,E)[-t],
	\]
	where $t = \dim R$ and $E =  E(R/\mathfrak{m})$ denotes the injective hull of the residue field. That is, 
	$H^i_{\mathfrak{m}}(X) \cong \Hom_R(\Ext^{t-i}_R(X, D^{\bullet}), E)$ for all $i \geq 0$ (see \cite{RD} resp. \cite[Theorem 12.2.1]{SS}).\\
	(C) If $(R,\mathfrak{m})$ is a local ring which is a factor ring
	of an $s$-dimensional Gorenstein local ring $(S,\mathfrak{n})$, and if $M$ is a finitely generated $R$-module, then 
	\[
	H^i_{\mathfrak{m}}(M) \cong \Hom_R(\Ext_S^{s-i}(M,S),\,E)
	\]
	for all $i \geq 0$. The
	$R$-module $$K^i(M) = \Ext_S^{s-i}(M,S), \, \, \, \, i = 0,\ldots,\dim_R M,$$ is called the {\it $i$-th module of deficiency} of $M$. We call $K(M) = K^d(M)$, with $d = \dim_R M,$  the {\it canonical module} of $M$ (see \cite[Section 1]{pS} for the basic properties of such modules).  By Matlis duality, it follows that
	\[\Hom_R(H^i_{\mathfrak{m}}(M), E)
	\cong \Ext_S^{s-i}(M,S) \otimes_R \hat{R}\]
	\noindent where $\hat{R}$ denotes the completion of $R$ (in the $\mathfrak{m}$-adic topology). 
	Moreover note that $K^i(M) \cong H^{t-i}(\Hom_R(M,D^{\bullet}))$.
\end{notation}

A particularly interesting case of \ref{cor-6} is when $\mathfrak{a}=\mathfrak{m}$, the maximal ideal of the local ring $R$, and $N$ is a Cohen-Macaulay $R$-module. 

\begin{corollary} \label{cor-7}
	Let $(R,\mathfrak{m})$ be a local ring and $M, N$ be two finitely
	generated $R$-modules, with $N$ Cohen-Macaulay. Let $t={\rm dim}_RN$. Then, there are isomorphisms
	\begin{itemize}
		\item[(a)]
		$\Ext_R^{i-t}(M,H^t_{\mathfrak{m}}(N)) \cong H^i_{\mathfrak{m}}(M,N)$, 
		\item[(b)]
		$\Tor_{i-t}^R(M,\, \Hom_R(H^t_{\mathfrak{m}}(N), E)) \cong \Hom_R(H^i_{\mathfrak{m}}(M,N), E)$, 
		\item[(c)]
		$H^i_{\mathfrak{m}}(M,K(\hat{N})) \cong \Ext_R^{i-t}(M,\Hom_R(N, E))
		\cong \Hom_R(\Tor_{i-t}^R(M,N), E)$, and 
		\item[(d)]
		$\Hom_R(H^i_{\mathfrak{m}}(M,K(\hat{N})), E) \cong \Tor_{i-t}^R(M, N) \otimes_R \hat{R}$,
	\end{itemize}
	for all $i \geq 0$.
\end{corollary}
\begin{proof}
	The results in (a) and (b) follow by  \ref{cor-6} and Matlis duality. 
	By Cohen's structure theorem, $\hat{R}$ is a factor ring of a Gorenstein ring.
	Therefore $K(\hat{N})$ exists and is a $t$-dimensional Cohen-Macaulay
	$\hat{R}$-module as well, with $\hat{N} \cong K(K(\hat{N}))$ (see, e.g., \cite[Theorem 1.14]{pS}) 
	and $H^t_{\mathfrak{m}}(K(\hat{N})) \cong  \Hom_R(N,E)$.
	By the statement in (a) there are isomorphisms
	\[
	\Ext_R^{i-t}(M,H^t_{\mathfrak{m}}(K(\hat{N}))) \cong H^i_{\mathfrak{m}}(M,K(\hat{N}))
	\]
	for all $i \geq 0$. This yields 
	\[
	H^i_{\mathfrak{m}}(M,K(\hat{N})) \cong \Ext_R^{i-t}(M,\Hom_R(N, E))
	\]
	 which proves the first
	part of (c). The rest is a consequence of adjointness and Matlis duality.
\end{proof}

A special case is when $R$ is Cohen-Macaulay and $N=R$. 

\begin{corollary} \label{cor-9}
	Let $(R,\mathfrak{m})$ be a $t$-dimensional Cohen-Macaulay local ring, and let $M$ be a finitely generated $R$-module. Then 
	\begin{itemize}
		\item[(a)] 	$H^t_{\mathfrak{m}}(M,K(\hat{R})) \cong \Hom_R(M, E)$, and $H^i_{\mathfrak{m}}(M,K(\hat{R})) = 0$ for all $i \not= t$.
		\item[(b)] If $\pd_RM < \infty$, then $H^t_{\mathfrak{m}}(M,R) \cong \Hom_R(M \otimes_R K(\hat{R}), E)$ 
		 and  $H^i_{\mathfrak{m}}(M,R) = 0$
		for all $i \not= t$.
	\end{itemize}
\end{corollary}

\begin{proof}
	The statements in (a) follow by \ref{cor-7}. In order to prove (b), first note that $H^i_{\mathfrak{m}}(M,R) \cong
	H^i_{\mathfrak{m}}(M, \hat{R})$ for all $i$. Now $K(\hat{R})$ is a
	$t$-dimensional Cohen-Macaulay module and $\hat{R} \cong K(K(\hat{R}))$.
	Because of \ref{cor-7}, we get
	\[
	H^i_{\mathfrak{m}}(M,\hat{R}) \cong \Hom_R(\Tor_{i-t}^R(M,K(\hat{R})),\,E)
	\]
	for all $i \geq 0$. In order to finish the proof, note that
	$\Tor_i^R(M,K(\hat{R})) = 0$ for all $i \not= 0$. This holds by virtue of \cite[Proposition 2.6]{rS}  (see also \cite[Lemma 2.2]{Yoshida}).
\end{proof}

For our next result we recall the notion of {\it Cohen-Macaulay defect} of a finitely generated  $R$-module $M$. This is the number $\delta_{\rm CM}(M)= {\rm dim}_R M -\depth_R M.$ Moreover we use the definition of generalized 
local cohomology of complexes as suggested in \ref{defnew-1}. 

\begin{theorem}\label{theo01}
	Let $(R,\mathfrak{m})$ denote a $t$-dimensional local ring possessing
	a dualizing complex $D^{\bullet}$. Let $M\not= 0$ be a finitely generated $R$-module, 
	and let $d=\dim_RM$. Then, 
	\begin{itemize}
		\item[(a)] $H^t_{\mathfrak{m}}(M,D^{\bullet}) \cong \Hom_R(M,E) \not= 0$, and 
		$H^i_{\mathfrak{m}}(M,D^{\bullet}) = 0$ for all $i \not= t$.
		\item[(b)] If $\pd_RM= p < \infty$, then  $H^{p+d}_{\mathfrak{m}}(M,  D^{\bullet})  \neq  0$ if and only if\, $\delta_{\rm CM}(M) = \delta_{\rm CM}(R)$.
	\end{itemize}
\end{theorem}

\begin{proof} 
	First note that $\RGam (D^{\bullet}) \cong E$. Since $M$ is a finitely generted $R$-module, we have isomorphisms 
	\[
	\RGam (\RHom_R(M,D^{\bullet})) \cong \RHom_R(M,\RGam (D^{\bullet})) \cong \Hom_R(M,E) [-t],
	\]
	which (together with \ref{thm-1}) proves (a). Now the statement in (b) follows by the  Auslander-Buchsbaum formula. \end{proof}

An immediate application is the following Cohen-Macaulayness characterization, which is a particular compelling reason to consider generalized local cohomology.

\begin{corollary}\label{CM-charact} Let $(R,\mathfrak{m})$ be a Cohen-Macaulay  local ring possessing
	a dualizing complex $D^{\bullet}$. Let $M$ be a  finitely generated  $R$-module with $d = \dim_RM$ and $\pd_RM=p < \infty$. Then the following assertions are equivalent:
\begin{itemize}
		\item[(a)] $M$ is Cohen-Macaulay,
		\item[(b)] $H^{p+d}_{\mathfrak{m}}(M, K(R))  \neq  0$, and
		\item[(c)] $H^{p+d}_{\mathfrak{m}}(M, D^{\bullet})  \neq 0$. 
\end{itemize}
\end{corollary}

Evidently, this corollary covers the case of a Gorenstein ring $R$ (where $K(R) \cong R$).

Now we are prepared for our local duality theorem for generalized
local cohomology, over local rings that are not necessarily Cohen-Macaulay. In particular, it  will be fundamental to our purposes in Section \ref{Aus-Rei}. 

\begin{theorem} \label{thm-3}
	Let $(R,\mathfrak{m})$ be a $t$-dimensional local ring possessing
	a dualizing complex $D^{\bullet}$. Let $M,N$ be two finitely generated $R$-modules, with $\pd_RM < \infty$. Then, there are
	isomorphisms
	\[
	H^i_{\mathfrak{m}}(M,N) \cong \Hom_R(\Ext_R^{t-i}(N,M\otimes_R^{\rm{L}} D^{\bullet}),  E)
	\]
	for all $i \geq 0$.
\end{theorem}

\begin{proof}
	Let $F_{\bullet}$ be a finite resolution of $M$ by finitely
	generated free $R$-modules. Then the Local Duality Theorem (see \ref{not-2}(B)) provides an isomorphism 
	\[
	\RGam (\RHom_R(M,N)) \cong \Hom_R(\Hom_R(\Hom_R(F_{\bullet},N),D^{\bullet}),E)[-t].
	\]
	Because $F_{\bullet}$ is a bounded complex of finitely generated $R$-modules, it follows that $\Hom_R(F_{\bullet},N) \cong \Hom_R(F_{\bullet},R) \otimes_R N$ and therefore 
	\[
	\Hom_R(\Hom_R(F_{\bullet},N),D^{\bullet}) \cong \Hom_R(N, \Hom_R(\Hom_R(F_{\bullet},R), D^{\bullet}))
	\]
	by adjointness. But now $$\Hom_R(\Hom_R(F_{\bullet},R), D^{\bullet})  \cong F_{\bullet} \otimes_R \Hom_R(R,  D^{\bullet}) \cong F_{\bullet} \otimes_R D^{\bullet}.$$ Putting the isomorphisms together, taking cohomology and using Lemma \ref{thm-1}, we get the claim. 
\end{proof}

The result above can be regarded as a generalization
of a theorem due to Suzuki (see \cite{nS}). As a corollary, which is also closely related to Herzog and Zamani  \cite[Theorem 2.1 (a)]{HZ}, we shall specialize it to the situation of a Cohen-Macaulay ring.

\begin{corollary} \label{cor}
	Let $(R,\mathfrak{m})$ be a $t$-dimensional Cohen-Macaulay local ring
	possessing a canonical module $K(R)$. Let $M,N$ be two finitely generated $R$-modules, with $\pd_RM < \infty$.
	Then, there are isomorphisms {\rm (}which are functorial in $N${\rm )}
	\begin{itemize}
		\item[(a)] $H^i_{\mathfrak{m}}(M,N) \cong
		\Hom_R(\Ext^{t-i}_R(N,M \otimes_R K(R)), E)$, and
		\item[(b)] $\Hom_R(H^i_{\mathfrak{m}}(M,N), E) \cong
		\Ext^{t-i}_R(N,M \otimes_RK(R)) \otimes_R \hat{R}$,
	\end{itemize}
	for all $i \geq 0$.
\end{corollary}

\begin{proof}
	First note the quasi-isomorphism $K(R) \cong D^{\bullet}$, since $R$ is a Cohen-Macaulay ring. Then there 
	is a quasi-isomorphism  $F_{\bullet} \otimes_R K(R) \cong F_{\bullet} \otimes_RD^{\bullet}$. But now $F_{\bullet} \otimes_R K(R)$ is a resolution of $M \otimes_R K(R)$ since $\Tor_i^R(M,K(R)) = 0$ for all $i > 0$ (see \cite[Proposition 2.6]{rS} or \cite[Lemma 2.2]{Yoshida}). This proves the statement as a consequence of \ref{thm-3}.
\end{proof}

By the works of Peskine and Szpiro \cite{PS} and Roberts \cite{pR}, the well-known Bass conjecture is known to be true. That is, a local ring possessing a finitely generated module of finite injective dimension is a Cohen-Macaulay ring. In view of this fact, we refer to the dual statement of \ref{cor} shown by Herzog and Zamani (see \cite[Theorem 2.1 (b)]{HZ}).

\begin{proposition} \label{cor1}
	Let $(R,\mathfrak{m})$ be a $t$-dimensional Cohen-Macaulay local ring
	possessing a canonical module $K(R)$. Let $M,N$ be two finitely generated $R$-modules, with $\id_RN < \infty$. Then, there are isomorphisms {\rm (}which are functorial in $M${\rm )}
	\begin{itemize}
		\item[(a)] $H^i_{\mathfrak{m}}(M,N) \cong
		\Hom_R(\Ext^{t-i}_R(\Hom_R(K(R),N),M ), E)$, and
		\item[(b)] $\Hom_R(H^i_{\mathfrak{m}}(M,N), E) \cong
		\Ext^{t-i}_R(\Hom_R(K(R),N),M ) \otimes_R \hat{R}$,
	\end{itemize}
	for all  $i \geq 0$.
\end{proposition}

We finish this section by recording duality in the special case of Gorenstein rings. 

\begin{corollary} \label{cor-10}
	Let $(R,\mathfrak{m})$ be a $t$-dimensional  Gorenstein local ring, and let
	$M, N$ be two finitely generated $R$-modules. Then
	\[
	H^i_{\mathfrak{m}}(M,N) \cong
	\Hom_R(\Ext^{t-i}_R(N,M ), E)
	\]
	for all $i \geq 0$, if $\pd_RM < \infty$ or $\pd_RN < \infty$. In addition, $$\sup\{i \in \mathbb{N} \, | \, H^i_{\mathfrak{m}}(M,N) \not= 0\} =
	\dim R - \grade(\Ann_RN,M).$$
\end{corollary}

\begin{proof}
	For a Gorenstein local ring $R$, we have $K(R) \cong R$.  Hence
	the first part follows by \ref{cor} and \ref{cor1}. For the non-vanishing, 
	we refer to \ref{rem-4}(A) in the next section.
\end{proof}

It is worth pointing out that the last claim of Corollary \ref{cor-10}  follows alternatively by \cite[Theorem 3.5]{DH}.

\section{Generalized Deficiency and Canonical Modules}
In \ref{not-2}(C) we used the notion of modules of deficiency as a
technical tool for a description of local cohomology modules.
Here we extend the definition in order to describe generalized local cohomology in the Gorenstein case.

\begin{definition} \label{def-2}
	Let $(R,\mathfrak{m})$ be a $t$-dimensional Gorenstein local ring, and let $M$ be a finitely generated $R$-module with $\pd_R M < \infty$.
	For a finitely generated $R$-module $N$, we define
	\[
	K^i(N,M) := \Ext_R^{t-i}(N,M) ,\;  i \in \mathbb{N},
	\]
	the {\it $i$-th module of deficiency of $N$ with respect to $M$}. Because of
	\ref{cor-10}, we have $H^i_{\mathfrak{m}}(M,N) \cong \Hom_R(K^i(N,M),  E)$. Note that for $M = R$ we recover the original definition of the
	$i$-th module of deficiency of $N$ (cf. \ref{not-2}(C)).
\end{definition}

In the following we will discuss some properties of the notion of
generalized modules of deficiency.

\begin{remark} \label{rem-4}
	(A) For a ring $R$ and two finitely
	generated $R$-modules $M, N$ with
	$\mathfrak{c} := \Ann_R N$, it follows that
	\[
	\grade (\mathfrak{c},M) = \inf \{i \in \mathbb{Z} \, | \, \Ext_R^i(N,M) \not= 0\}
	\]
	(see \cite[1.2.10]{BH}). We denote this integer  by $\grade(N,M)$. It is
	equal to the length of a maximal $M$-regular sequence contained in $\Ann_RN$.
	Therefore, with the notation of \ref{def-2}, $$\sup \{i \in \mathbb{Z} \, | \, K^i(N,M)\not= 0\} = \dim R -
	\grade(N,M).$$
	Note that $\grade(N,R) = \grade(N)$ is the usual definition of grade.\\
	(B) Now let $(R,\mathfrak{m})$ be a Gorenstein local ring. If $\pd_R M < \infty$ as in \ref{def-2}, it follows that $\id_R M < \infty$, with $\id_R M \leq t = \dim R$.
	Therefore $K^i(N,M) = 0$ for all $i < 0$ and $i > t$. Moreover,
	\[
	\Ass_R K^t(N,M) = \Ass_R \Hom_R(N,M)= \Ass_R M \cap \Supp_R N,
	\]
	and  $K^t(N,M) = 0$ if and only if $\Ass_R M \cap \Supp_R N$ is empty.\\
	(C) Let $\mathfrak{p}$ denote a prime ideal of a Gorenstein local ring
	$R$. Then $\dim R = \dim R/\mathfrak{p} + \dim R_{\mathfrak{p}}$
	and therefore
	\[
	K^i(N,M)\otimes_R R_{\mathfrak{p}} \cong K^{i-\dim R/\mathfrak{p}}(N_{\mathfrak{p}},M_{\mathfrak{p}})
	\]
	for all $i$. Hence we get $\dim K^i(N,M) \leq i$ for all $i$. This is a rather row estimate. Note that if $M \otimes_R N$ has finite length then $\dim_R K^i(N,M) = 0$ for all $i \in \mathbb{N}$, which follows from the general inclusion $\Supp_R \Ext_R^j(M,N) \subseteq \Supp_R M\otimes_RN$ .\\
	(D) As follows by (A) and (B), we know that $K^i(N,M) = 0$ for $i > s$ and
	$K^s(N,M) \not= 0$ for $s = t - \grade(N,M)$. In case $M = R$ (i.e.,
 ordinary local cohomology), we have $s= t - \grade(N)$, and note that $$\Ext_R^{t-\grade(N)}(N,R) = K^{\dim_R N}(N)$$ (by \ref{cor-10} with $M=R$) is precisely the ordinary canonical module
	of $N$. This motivates us to consider the definition below.
\end{remark}

\begin{definition}\label{def-3}
	Let $(R,\mathfrak{m})$ be a $t$-dimensional Gorenstein local ring. Let
	$M, N$ be two finitely generated $R$-modules, with $\pd_R M < \infty$. Then
	$$K^s(N,M): = \Ext_R^g(N,M),$$
	\noindent with $g = \grade(N,M)$ and $s = t-g$, is called {\it the
	generalized canonical module of $N$ with respect to $M$}.
	For $M = R$, we have
	$K^s(N,R)=K(N)$, which shows that our definition is a generalization of
	the usual notion for a single module.

	Furthermore, there is a canonical isomorphism
	\[
	H^s_{\mathfrak{m}}(M,N) \cong \Hom_R(K^s(N,M), E)
	\]
	as follows by the generalized local duality in the Gorenstein case (\ref{cor-10}).
\end{definition}

Our next result concerns the behavior of the generalized canonical module with respect to localization.

\begin{proposition} \label{prop-3} 	Let $(R,\mathfrak{m})$ be a $t$-dimensional Gorenstein local ring, and let $M, N$ be two finitely generated $R$-modules. Suppose that $\pd_RM < \infty$. Let $g = \grade(N,M)$
	and $s = t -g$. Then  $g = \grade(N_{\mathfrak{p}},M_{\mathfrak{p}})$ and
	\[
	K^s(N,M)\otimes_R R_{\mathfrak{p}} \cong K^{s-\dim R/\mathfrak{p}}(N_{\mathfrak{p}}, M_{\mathfrak{p}}) \not= 0
	\]
	for all $\mathfrak{p} \in \Supp_R K^s(N,M)$.
\end{proposition}

\begin{proof}
	First note that $g= \grade(N,M) = \grade(N_{\mathfrak{p}},M_{\mathfrak{p}})$ for $ \mathfrak{p} \in \Supp K^s(N,M)$. By localization, we get
	\[
		\Ext_R^s(N,M)  \otimes_R R_{\mathfrak{p}} \cong \Ext_{R_\mathfrak{p}}^s(N_{\mathfrak{p}},M_{\mathfrak{p}}) \not= 0.
	\]
	Then the claim follows by \ref{rem-4}(D).
\end{proof}

A more detailed discussion on the structure of $K^s(N,M)$ is given in the following results.

\begin{proposition} \label{lem-1}
	Let $(R,\mathfrak{m})$ be a  $t$-dimensional Gorenstein local ring, and let $M, N$ be two finitely generated $R$-modules. Assume that $\pd_RM < \infty$.
Let $g = \grade(N,M)$ and $s = t-g$. Then:
	\begin{itemize}
		\item[(a)] $\Ass K^s(N,M) = \Supp_RN \cap \Ass_R M/\xx M,$ where
		$\xx = x_1,\ldots, x_g$ denotes an $M$-regular sequence contained in $\Ann_RN$.
		\item[(b)] $\dim_R K^s(N,M) = \dim_R M \otimes_RN$.
		\item[(c)] $\depth_R K^s(N,M) > 0$ if and only if
		$s > \pd_R M$.
	\end{itemize}
\end{proposition}

\begin{proof}
	In the case of $g = 0$, the claim of (a) is clear by \ref{rem-4}(B).
	Now we choose $x \in \Ann_RN$ an $M$-regular element. Then the short exact sequence
	\[0 \to M \stackrel{x}{\longrightarrow} M \to M/xM \to 0\]
	\noindent yields
	$g = \grade(N,M) = \grade(N,M/xM) +1$ as well as an isomorphism
	\[
	\Ext_R^g(N,M) \cong \Ext_R^{g-1}(N,M/xM),
	\]
	as follows by the long exact cohomology sequence $\Ext_R^{\cdot}(N,\cdot)$.
	Therefore $K^s(N,M) \cong K^{s+1}(N,M/xM)$. By iterating this argument,
	it follows that
	\[
	\Ext_R^g(N,M) \cong \Hom_R(N,M/\xx M),
	\]
	where $\xx = x_1,\ldots,x_g$ denotes a maximal $M$-regular sequence contained
	in $\Ann_R N$. This proves the claim in (a).
	
	By view of the equality in (a), it follows that
	\[
	\Supp_R K^s(N,M) = \Supp_R N \cap \Supp_R M/\xx M = \Supp_R N \cap \Supp_R M
	\cap V(\xx R).
	\]
	Since $\xx R \subseteq \Ann_RN$, we have that $\Supp_R N \subseteq V(\xx R)$.
	Therefore $\Supp_R K^s(N,M) = \Supp_R M \otimes_R N$, as required.
	
	For the proof of (c), note that $\depth K^s(N,M) > 0$ if and only if
	\[
	\mathfrak{m} \not\in \Supp_R N \cap \Ass_R M/\xx M,
	\]
	i.e., if and only if
	$\depth_R M/\xx M = \depth_R M -g > 0$ for a maximal $M$-regular sequence
	$\xx = x_1,\ldots,x_g$ contained in $\Ann_RN$. The Auslander-Buchsbaum formula provides $\depth_R M = t - \pd_RM$, which proves (c).
\end{proof}

Note that, for $M = R$, we have $\depth K(N) > 0$ if and only if $\dim_R N > 0$.
This is the classically known result (see \cite{pS}).

\begin{proposition} \label{prop-4}
	With the notations and setting of \ref{lem-1}, suppose $s > \pd_RM$ and let $N' = N/H^0_{\mathfrak{m}}(N)$. Then $K^s(N,M) \cong K^s(N',M)$.
\end{proposition}

\begin{proof}
	First note that $s > \pd_RM = t- \depth_RM$ yields that $\dim_RN >0$. Therefore
	$\grade(N,M) = \grade(N',M)$, because $\Supp_RN = \Supp_R N'$.
	Moreover $$\Ext_R^i(H^0_{\mathfrak{m}}(N),M) = 0 \mbox{ for all } i < \depth_R M,$$
	because $H^0_{\mathfrak{m}}(N)$ has support in $\{\mathfrak{m}\}$. Now the
	short exact sequence
	\[0 \to H^0_{\mathfrak{m}}(N) \to N \to N' \to 0\]
	\noindent yields -- by applying $\Ext_R^{\cdot}(\cdot,M)$ -- the isomorphism
	\[
	\Ext_R^g(N,M) \cong \Ext_R^g(N',M),
	\]
	because $g < \depth_RM$, by the assumption $s > \pd_RM$ and the Auslander-Buchsbaum formula.
\end{proof}

Let $k \geq 1$ denote an integer. For the next result, recall that an $R$-module $M$ satisfies the condition $S_k$ if
\[
\depth M_{\mathfrak{p}} \geq \min \{k,\,\dim M_{\mathfrak{p}}\} \;
\mbox{ for all }\, \mathfrak{p} \in \Supp_R M.
\]

A slight improvement of \ref{lem-1}(c) is the following proposition.

\begin{proposition} \label{prop-5}
	With the notations and setting of \ref{lem-1}, suppose $s > \pd_RM +1$.
	Then $K^s(N,M)$ satisfies condition $S_2$. In particular, $\depth_R K^s(N,M) > 1.$
\end{proposition}

\begin{proof}
	We have $g= \grade(N,M) = \grade(N_{\mathfrak{p}},M_{\mathfrak{p}})$ for $\mathfrak{p} \in \Supp K^s(N,M)$.
	By the assumption, we get $g < \depth_RM +1$. Because of
	\[
	\depth_R M \leq \depth_{R_{\mathfrak{p}}} M_{\mathfrak{p}} + \dim R/\mathfrak{p} \mbox{ for all } \mathfrak{p} \in \Supp_RM
	\]
	(see the consequence of \cite[1.10]{pS}), it follows that
	\[
	\grade(N_{\mathfrak{p}},M_{\mathfrak{p}}) -\dim R/\mathfrak{p} < \depth M_{\mathfrak{p}} +1.
	\]
	By view of \ref{rem-4}(D), this yields the  isomorphism
	$
	K^s(N,M) \otimes_R R_{\mathfrak{p}} \cong K^s(N_{\mathfrak{p}},M_{\mathfrak{p}})$.
	In order to prove the claim, it will be enough to show that
	\[
	\depth K^s(N,M) \geq \min \{2,\,\dim_R K^s(N,M)\}.
	\]
	If $ \dim_R K^s(N,M) = 0$, there is nothing to show. In the case
	$\dim_R K^s(N,M) = 1$, the claim follows by \ref{lem-1}(c). So assume
	$ \dim_R K^s(N,M) \geq 2$.
	In virtue of \ref{lem-1} and \ref{prop-4}, we have that $K^s(N,M) \cong
	K^s(N',M)$ and $$\depth_R K^s(N',M) >0,$$ where $N' = N/H^0_{\mathfrak{m}}(N)$.
	That is, without loss of generality we may assume $\depth_RN >0$. Therefore we may choose an element $x \in \mathfrak{m}$ that is $N$-regular and also
	$K^s(N,M)$-regular (see \ref{lem-1}(c)).
	The short exact sequence \[0 \to N \stackrel{x}{\longrightarrow} N\to N/xN
	\to 0\] \noindent induces an exact sequence
	\[
	0 \to \Ext_R^g(N/xN,M) \to \Ext_R^g(N,M) \stackrel{x}{\longrightarrow}
	\Ext_R^g(N,M) \to \Ext_R^{g+1}(N/xN,M).
	\]
	Since $x$ is $\Ext_R^g(N,M)$-regular, it follows that $\Ext_R^g(N/xN,M) = 0$. Since $\Ext_R^g(N,M)$ does not vanish, Nakayama's lemma implies that
	$\Ext_R^{g+1}(N/xN,M) \not= 0$ and therefore $\grade(N/xN,M) = g+1$.
	The exact sequence yields an injection
	\[
	0 \to  \Ext_R^g(N,M)/x \Ext_R^g(N,M) \to  \Ext_R^{g+1}(N/xN,M).
	\]
	By \ref{lem-1},  $ \Ext_R^{g+1}(N/xN,M)$ has positive
	depth. By the injection, we conclude that $\depth_R \Ext_R^g(N,M) \geq 2$.
\end{proof}

\section{Application: Prescribed Bound For Projective Dimension}\label{Aus-Rei}

Our main result in this section is the following application of our general local duality theorem (\ref{thm-3}), concerning the detection of a prescribed upper bound for the projective dimension of a module satisfying suitable cohomological conditions (similar results, in the particular case of Cohen-Macaulay rings with canonical module, can be found in \cite{Cle-Vic}). We tacitly maintain the setting of \ref{not-2}.

\begin{theorem}\label{main-result} Let $(R,\mathfrak{m})$ be a local ring with dimension $t$ and depth $r\geq 1$, possessing a dualizing complex $D^{\bullet}$. Let $M$ and $N$ be finitely generated $R$-modules with $\pd_RM < \infty$ and $\pd_RN < \infty$, satisfying
	$$\Ext_R^{j}(N,M\otimes_R^{\rm{L}} D^{\bullet}) \, = \, 0, \, \, \, \, j=t-r+i, \ldots, t$$
	for some positive integer $i \leq r$. Then, $\pd_R N<i$.
\end{theorem}
\begin{proof} Since $\pd_RM$ is finite, \ref{thm-3} gives isomorphisms
	\[
	H^n_{\mathfrak{m}}(M,N) \, \cong \, \Hom_R(\Ext_R^{t-n}(N, M\otimes_R^{\rm{L}} D^{\bullet}),  E)
	\]
	for all $n \geq 0$. Since, by hypothesis, $\Ext_R^{j}(N, M\otimes_R^{\rm{L}} D^{\bullet})$ vanishes for $j=t-r+i, \ldots, t$, we get   
$$H^n_\mathfrak{m}(M, N)   \, = \, 0$$ for all $n=0,\ldots, r-i$. Further, because ${\rm depth}_RN$ is the least integer $l\geq 0$ for which $H^l_{\mathfrak{m}}(M,N)\neq 0$ by \ref{prop-1} (see also \cite[Theorem 2.3]{nS}), we conclude that ${\rm depth}_R N> r-i$. Finally, we apply the Auslander-Buchsbaum formula.
\end{proof}

\begin{corollary}\label{main-cor} Let $(R,\mathfrak{m})$ be a local ring with dimension $t$ and depth $r\geq 1$, possessing a dualizing complex $D^{\bullet}$. Let $N$  be a finitely generated $R$-module with $\pd_RN < \infty$. For a positive integer $i\leq r$, assume any of the following conditions:
	\begin{itemize}
		\item[(a)] $\Ext_R^{j}(N, D^{\bullet})=0$\, for\, $j=t-r+i,\ldots, t$;
		
		\item[(b)] $\Ext_R^{j}(N, N\otimes_R^{\rm{L}} D^{\bullet})=0$ \, for\, $j=t-r+i,\ldots, t$.
		
	\end{itemize}
Then, $\pd_R N<i$.
\end{corollary}
\begin{proof} Item (a) (resp. (b)) follows readily by taking $M=R$ (resp. $M=N$) in \ref{main-result}.
\end{proof}

\begin{remark}\label{Jo-rem} We point out that \ref{main-cor}(b) is somewhat related to an interesting question raised by Jorgensen (see \cite[Question 2.7]{jor}). Precisely, it was asked whether $\pd_RN<\infty$ and ${\rm Ext}_R^n(N, N)=0$ imply $\pd_RN<n$, for some positive integer $n$, where the local ring $R$ is assumed to be a complete intersection of positive codimension. In our result, $R$ can be taken  much more general but on the other hand we need the vanishing of more cohomology modules which, in addition, involve tensoring with $D^{\bullet}$. Below we propose a generalization of Jorgensen's problem.

\end{remark}

\begin{question}\label{Jo-ques} Let $(R,\mathfrak{m})$ be a local ring with positive depth, possessing a dualizing complex $D^{\bullet}$ and let $N$ be a finitely generated $R$-module with $\pd_RN < \infty$. If $$\Ext_R^{n}(N,\, N\otimes_R^{\rm{L}} D^{\bullet}) \, = \, 0$$ for some positive integer $n$, is it true that $\pd_R N<n$\,? In particular, we suggest to consider Jorgensen's question with $R$ Gorenstein.

\end{question}

Next, we turn to the derivation of freeness criteria as well as some observations and questions. In \ref{main-result} and \ref{main-cor}, the case $i=1$ yields immediately the following results.

\begin{corollary}\label{cor1-free} Let $(R,\mathfrak{m})$ be a local ring with dimension $t$ and depth $r\geq 1$, possessing a dualizing complex $D^{\bullet}$. Let $M$ and $N$ be finitely generated $R$-modules with $M\neq 0$,  $\pd_RM < \infty$ and $\pd_RN < \infty$, satisfying
	$$\Ext_R^{j}(N,M\otimes_R^{\rm{L}} D^{\bullet}) \, = \, 0, \, \, \, \, j=t-r+1, \ldots, t.$$
	Then, $N$ is a free $R$-module.
\end{corollary}

\begin{corollary}\label{cor2-free} Let $(R,\mathfrak{m})$ be a local ring with dimension $t$ and depth $r\geq 1$, possessing a dualizing complex $D^{\bullet}$. Let $N$  be a finitely generated $R$-module with $\pd_RN < \infty$. Assume any of the following conditions:
	\begin{itemize}
		\item[(a)] $\Ext_R^{j}(N, D^{\bullet})=0$\, for\, $j=t-r+1,\ldots, t$;
		\item[(b)] $\Ext_R^{j}(N, N\otimes_R^{\rm{L}} D^{\bullet})=0$ \, for\, $j=t-r+1,\ldots, t$.
	\end{itemize}
Then, $N$ is a free $R$-module.
\end{corollary}

\begin{example}\rm We want to illustrate that in \ref{cor2-free} the condition $\pd_RN<\infty$ cannot be removed, at least in the case of a 1-dimensional complete intersection local ring. Let $\Bbbk$ be a field, $R=\Bbbk [[x, y]]/(xy)$ and $N=R/xR$. Then we have
        \[
		{\rm Ext}^1_R(N, R) \, \cong \, 0:_R(0:_R x)/xR \, = \, (0:_Ry)/xR \, = \, 0,
		\] and
        \[
		{\rm Ext}^1_R(N, N) \, = \, {\rm Ext}^1_R(R/xR, N) \, \cong \, 0:_N(0:_R x)/xN \, = \, 0:_Ny \, = \, 0,
		\]
but clearly $N$ is not $R$-free.
\end{example}

\begin{remark}\label{AR-rem} Let us recall the celebrated Auslander-Reiten conjecture, originally raised in \cite{auslander} about 45 years ago. If $R$ is a local ring and $N$ is a finitely generated $R$-module satisfying
$${\rm Ext}^i_R(N, R) \, = \, {\rm Ext}^i_R(N, N) \, = \, 0$$ for every $i>0$, then the conjecture predicts that $N$ must be free. The problem remains open in this general form, but it has been settled affirmatively in some important cases; we refer, e.g., to \cite{araya}, \cite{sgoto}, \cite{hule} and their suggested references. What we want to point out herein is that the freeness criteria established in our \ref{cor2-free} can be said to be of {\it Auslander-Reiten type} due to the similarity with the hypotheses of the conjecture (even though our result requires the vanishing of finitely many cohomology modules, while of course the condition $\pd_RN<\infty$ plays a key role). This, in turn, inspires us to propose the following Auslander-Reiten type problem:

\end{remark}

\begin{question}\label{AR-ques} Let $R$ be a local ring possessing a dualizing complex $D^{\bullet}$ and let $N$ be a finitely generated $R$-module such that $${\rm Ext}^i_R(N, D^{\bullet}) \, = \, {\rm Ext}^i_R(N, N\otimes_R^{\rm{L}} D^{\bullet}) \, = \, 0$$ for every $i>0$. Is it true that $N$ must be free? Setting $t=\dim R$ and $r={\rm depth}\,R$ as in \ref{cor2-free}, does $i> t-r$ suffice? If $R$ is assumed to be Cohen-Macaulay with a canonical module $K(R)$, then our question asks whether $N$ is free provided that $${\rm Ext}^i_R(N, K(R)) \, = \, {\rm Ext}^i_R(N, N\otimes_RK(R)) \, = \, 0$$ for all $i>0$. Notice that, in case $R$ is Gorenstein, this question matches exactly the original statement of the Auslander-Reiten conjecture.

\end{question}

\bigskip

\noindent{\bf Acknowledgements.} The first three authors were partially supported by  CNPq-Brazil grant 421440/2016-3. The third-named author was also supported by CNPq-Brazil grant 301029/2019-9 and by FAPESP-Brazil grant 2019/21843-2. The authors also wish to express their gratitude to the anonymous referees for their contributions in improving the paper, for instance the conceptual point of view suggested in \ref{defnew-1}.

\end{document}